\documentclass[14pt]{amsart}

\usepackage{bm}

 \usepackage{upgreek}

\usepackage{stmaryrd} 
\usepackage[T1]{fontenc}
\usepackage{textcomp}

\usepackage{amsopn, amsthm, amsgen, amscd,amsmath}
\usepackage[bbgreekl]{mathbbol}
\usepackage[mathcal]{euscript}
\DeclareMathAlphabet{\mathbbm}{U}{bbm}{m}{n}
\usepackage{color}
\usepackage{tikz}
  \usepackage{graphicx}
\usepackage{epstopdf}
\usepackage[small,nohug,heads = LaTeX]{diagrams} \diagramstyle[labelstyle=\scriptstyle]

\definecolor{CadetBlue}{cmyk}{0.62, 0.57, 0.23, 0 }
\definecolor{black}{cmyk}{1, 0.5, 0, 0 }
\definecolor{RedViolet}{cmyk}{0.07, 0.9, 0, 0.34 }
\definecolor{SeaGreen}{cmyk}{0.69, 0, 0.5, 0}

\definecolor{RoyalBlue}{cmyk}{1, 0.5, 0, 0 }

\definecolor{RedViolet}{cmyk}{0.07, 0.9, 0, 0.34 }

\definecolor{BrickRed}{cmyk}{0, 0.89, 0.94, 0.28}

\DeclareMathAlphabet{\mathpzc}{OT1}{pzc}{m}{it}

\newcommand{\R}{\mathbb R}
\newcommand{\C}{\mathbb C}

\newcommand{\F}{\mathbb F}

\newcommand{\Q}{\mathbb Q}
\newcommand{\Z}{\mathbb Z}

\newcommand{\T}{\mathbb T}

\newcommand{\SI}{\mathbb S}

\newtheorem{theo}{Theorem}
\newtheorem{lemm}{Lemma}
\newtheorem{prop}{Proposition}
\newtheorem{coro}{Corollary}

\newtheorem*{theon}{Theorem}
\newtheorem*{coronn}{Corollary}
\newtheorem*{coron}{Characterization Theorem}

\theoremstyle{definition}
\newtheorem{defi}{Definition}

\theoremstyle{remark}

\newtheorem{exam}{Example}

\newtheorem{note}{Note}

\title[Alexander Polynomial]{On The Characterization Problem of Alexander Polynomials of Closed 3-Manifolds}
\author{K. Alcaraz}
\address{The University of Texas at Austin, Department of Mathematics, 1 University Station C1200 Austin, TX 78712-0257}
\email{karin@math.utexas.edu}
\date{April 15, 2014}
\begin{document}
\vspace{2cm}
\begin{abstract} 

\noindent We give a characterization for the Alexander polynomials $\Delta_{M}$ of closed orientable 3-manifolds $M$ with first Betti number $b_{1}M=1$ and some partial results for the characterization problem in the cases of $b_{1}M>1$.  We first prove an analogue of a theorem of Levine: that the product of an Alexander
polynomial $\Delta_{M}$ with a non-zero trace symmetric polynomial in $b_{1}M$ variables is again an Alexander polynomial
of a closed orientable 3-manifold.  Using the fact that there exist $M$ with $\Delta_{M}=1$ for $b_{1}M=1,2,3$, we conclude
that the symmetric polynomials of non-zero trace in 1, 2 or 3 variables are Alexander polynomials of closed orientable 3-manifolds.
When  $b_{1}M=1$ we prove  that non-zero trace symmetric polynomials are the only ones that can arise.  Finally,
for $b_{1}M\geq 4$, we prove that $\Delta_{M}\not=1$ implying that for such manifolds not all non-zero trace symmetric polynomials occur.

 \end{abstract}
 
 \maketitle
\tableofcontents
\section{Introduction}

Let $M$ be a manifold, $b_{1}M$ its first Betti number.  For $b_{1}M\geq 1$, the Alexander polynomial of $M$, $\Delta_{M}$, is a polynomial  invariant in $b_{1}M$ variables
first defined by Alexander \cite{Alex} for $M= \SI^{3}-K$ a knot complement and more generally for $M$ having  finitely presented $\pi_{1}$ by Fox \cite{FoxI}. 

 In the case of knot complements, a characterization of Alexander polynomials was given by Seifert \cite{Ro}, and later,
in the case of link complements, necessary conditions on Alexander polynomials were described by Torres \cite{To}  (when $b_{1}=2$, these conditions are insufficient \cite{Hill1}, \cite{Pl}).  

In this paper, we consider the characterization problem for 
Alexander polynomials of closed orientable 3-manifolds, giving a characterization in the case $b_{1}M=1$ and some partial results concerning characterization in the cases of $b_{1}M> 1$.  

Our first result is the following analog for closed orientable 3-manifolds of a Theorem of Levine (see \S \ref{LevineSection} 
or \cite{Le2}), which allows us to produce by multiplication new Alexander polynomials from a given one.
By the {\it trace} of a polynomial we mean the sum of its coefficients.

\begin{theon} Let $M$ be a closed orientable $3$-manifold with $b_{1}M=n$ and let $\lambda$ be a symmetric Laurent polynomial in $n$ variables with non zero trace.
Then there exists a $3$-manifold $M'$ with $b_{1}M'=n$ for which
\[ \Delta_{M'}=\lambda\cdot\Delta_{M}.\]
\end{theon}

This Theorem is proved in \S \ref{LevineSection}.

For low Betti numbers $b_{1}=1, 2$ or $3$, the following closed orientable 
manifolds have 
$\Delta_{M}=1$:

\begin{itemize}
\item[-] $\SI^{1}\times\SI^{2}$, $b_{1}=1$.
\item[-] $\text{\sf H}_{3}(\R)/\text{\sf H}_{3}(\Z )$ = Heisenberg manifold \cite{Mc}, $b_{1}=2$. 
\item[-]  $\T^{3}$ = the 3-torus, $b_{1}=3$.
\end{itemize}

Combining these examples with the above analog of Levine's Theorem we have

\begin{coronn}  Every symmetric Laurent polynomial in $1$, $2$ or $3$ variables having non-zero trace is the Alexander polynomial of a $3$-manifold with first Betti number $1$, $2$ or $3$.
\end{coronn}

It is natural to ask if the statement in the Corollary gives necessary conditions for a characterization for Betti numbers $b_{1}=1, 2$ or $3$.  In the case $b_{1}=1$, it does.
We say that a Laurent polynomial is {\it unit symmetric} if it is symmetric after multiplication with a unit of the ring
$\Z [t_{1}^{\pm 1},\dots ,t_{n}^{\pm 1}]$.

\begin{coron}  Let $\lambda$ be a Laurent polynomial in $1$ variable.  Then $\lambda =\Delta_{M}$
for some closed $3$-manifold $M$ with $b_{1}M=1$ $\Leftrightarrow$ $\lambda$ is unit symmetric and has non-zero trace.
\end{coron}

Finally, we consider the case of manifolds with $b_{1}\geq 4$.  Our main result is the following:

 \begin{theon} $\Delta_{M}\not =1$ for any closed $3$-manifold with $b_{1}M\geq 4$.  
 \end{theon}
 
  A consequence of this is the following negative result: not every symmetric polynomial occurs as the Alexander polynomial
 for a 3-manifold having $b_{1}\geq 4$.

\vspace{3mm}

\noindent {\bf Acknowledgements:} This paper is based on results contained in my PhD thesis, which was supervised by Marc Lackenby. I would also like to thank Cameron Gordon for helpful conversations.

\section{Alexander Polynomial}

From this point on, all 3-manifolds will be assumed to be orientable.

 We begin by reviewing a definition of the Alexander polynomial which employs absolute homology \cite{Ma}, \cite{Ro}.  This definition differs from those that appear in \cite{Blan,FoxI,Mc,Turaev0}  which use either relative homology or Fox calculus.   A proof of the equivalence of all three definitions can be found in \cite{me} (the equivalence is in fact implicit in Theorem 2.7 of \cite{Ma} as well as Theorem 16.5 of \cite{Turaevbook}).

Let $\psi :G\rightarrow H$ be an epimorphism of a finitely presented group $G$ onto a finitely generated free abelian group $H$ of rank $n$.  Denote by $\Z [H]$ the group ring associated to $H$ i.e.\ the ring of formal
finite linear combinations
\[ n_{1}h_{1}+\dots +n_{k}h_{k},\quad n_{i}\in\Z,\; h_{i}\in H,\]
and the product in $\Z [H]$ is defined by linear extension of the product of $H$.  Note that if we choose a basis
$t_{1},\dots ,t_{n}$ of $H$ then
we may write elements of $H$ as
\[  t^{I} = t_{1}^{i_{1}}\cdots t_{n}^{i_{n}} \] for $i_{1},\dots ,i_{n}\in\Z$, so that the elements of  $\Z [H]$ may be viewed as a multivariable Laurent polynomials in the 
multivariable $t$.

Let $(X,x)$ be a pointed $CW$-complex whose $0$-skeleton consists of the base point $x$, such that $G=\pi_{1}(X,x)$.  Let 
\[ p_{\psi} \colon \widehat{X}_{\psi}\longrightarrow X\] be the normal covering space corresponding to $\psi$: that is, the covering indexed by ${\rm Ker}(\psi)\subset G$ with deck group $G/{\rm Ker}(\psi)\cong H$. 

 The {\bf {\em Alexander module}} is defined to be 
\begin{equation}\label{firstdefi}
A_{\psi}=H_{1}(\widehat{X}_{\psi})=H_{1}(\widehat{X}_{\psi};\;\mathbb{Z}),\end{equation}
where its structure of $\Z [H]$-module comes from the action of $H$ on $\widehat{X}_{\psi}$ by deck transformations.

For any finitely presented module $A$ over $\mathbb{Z}[H]$ consider a free resolution 
 \[ \mathbb{Z}[H]^{m} \stackrel{P}{\longrightarrow}  \mathbb{Z}[H]^{n}\longrightarrow  A
\longrightarrow 0 .\]
Such a resolution may be defined using a presentation $A=\langle x_{1},\dots x_{n}|r_{1},\dots ,r_{m}\rangle$:
we take $ \mathbb{Z}[H]^{n}=\langle x_{1},\dots , x_{n}| \;\;\rangle$, $\mathbb{Z}[H]^{m} = \langle R_{1},\dots ,R_{m}| \;\;\rangle$
with $P(R_{i}) =r_{i}$. 
Note that we may assume, without loss of generality, that $m\geq n$.
One can represent $P$ by an $m\times n$ matrix
also denoted $P$.  
For each $d=0,\dots , n$ we define the {\bf {\em dth elementary ideal}} $I_{d}(A)\subset \mathbb{Z}[H]$ to be the ideal generated by the $(n-d)\times (n-d)$ minors of the matrix $P$.  These ideals are independent of the resolution of $A$
and form a chain
\[   I_{0}(A)\subset I_{1}(A)\subset \cdots \subset I_{n}(A)=\mathbb{Z}[H] .\]
The {\bf {\em dth order ideal}} is the smallest principal ideal containing $I_{d}(A)$, where any generator $\Delta_{d}(A)$
of it is called a {\bf {\em dth order}} of $A$; it is well-defined up to multiplication by units.  The $d$th order can also be defined as the greatest
common divisor of the $(n-d)\times (n-d)$ minors of the matrix $P$ (which is well-defined since $\mathbb{Z}[H]$ is a unique factorization domain).

When $A=A_{\psi}$ we denote by 
\begin{itemize}
\item $P_{\psi}$ any presentation matrix and call it
an {\bf {\em Alexander matrix}}.  
\item $I_{\psi}$ the $0$th elementary ideal of a presentation matrix $P_{\psi}$ and call it the {\bf {\em Alexander ideal}}.
\item $\Delta_{\psi}$ a $0$th order of $A_{\psi}$ and call it an {\bf {\em Alexander polynomial}}.
\end{itemize}

Let $M$ be a compact manifold with $\pi_{1}M$ finitely presented and with first Betti number $b_{1}M=n$.  
Let $H(M)=H_{1}M/{\rm Tor}(H_{1}M)\cong\Z^{n}$, where ${\rm Tor}(H_{1}M)$ is the torsion subgroup of $H_{1}M$, and consider the epimorphism
\[ \psi^{\text{fr-ab}}: \pi_{1}M\longrightarrow H_{1}M\longrightarrow H(M),\]
where the first map is the abelianization and the second map is the projection.
The normal covering associated to $K^{\text{fr-ab}}={\rm Ker}(\psi^{\text{fr-ab}})$
\[  p:\widehat{M}\longrightarrow M \]
is called the {\bf {\em universal free abelian cover}} of $M$.



\begin{prop}  Let $H$ be a finitely generated free abelian group and let $\psi :\pi_{1}M\rightarrow H$ be an epimorphism
with associated covering $p_{\psi}: \widehat{M}_{\psi}\rightarrow M$.  Then there exists a covering
\[  q_{\psi}:\widehat{M}\longrightarrow   \widehat{M}_{\psi}  \]
such that $p_{\psi}\circ  q_{\psi}=p$.
\end{prop}

\begin{proof}  Let $K_{\psi}= {\rm Ker}(\psi )$ which is the image of $\pi_{1} \widehat{M}_{\psi}$ by 
$(p_{\psi})_{\ast}$.  Since $H_{1}M$ is the abelianization
of $\pi_{1}M$, there is an epimorphism $r:H_{1}M\rightarrow H$ with $\psi= r\circ \psi^{\text{fr-ab}}$.
But $H$ is torsion free, therefore
all torsion elements of $H_{1}M$ are mapped to $0$.  This means that there exists
an epimorphism $r' :H(M)\rightarrow H$ so that $\psi= r'\circ \psi^{\text{fr-ab}}$.  We conclude
that $K^{\text{fr-ab}}\subset K_{\psi}$, and the Proposition follows from this.
\end{proof}

\begin{coro}\label{Alexunique} Any cover of $M$ with deck group $H(M)$ is isomorphic to the universal free abelian cover $p:\widehat{M}\longrightarrow M$.
\end{coro}

When $\psi=\psi^{\rm fr-ab}$ we denote $A_{M}$, $I_{M}$ and $\Delta_{M}$,  and refer to them as the {\bf {\em Alexander module, ideal}} and {\bf {\em polynomial}}
of $M$.  By Corollary \ref{Alexunique}, these invariants depend only on $M$.

\section{An Analog of a Theorem of Levine for Closed 3-manifolds}\label{LevineSection}

First we define three notions of symmetry that will be relevant for us.  

Let $H\cong\Z^{n}$, denote $\Lambda=\Z [H]$
the group ring and by $\Lambda^{\times}$ the group of units.
The inversion map in $H$, $h\mapsto h^{-1}$, extends to an automorphism $\iota:\Lambda\rightarrow\Lambda$.  

We say that a polynomial
$f\in\Lambda$ is 
\begin{itemize}
\item {\bf {\em symmetric}} if $\iota (f)=f$.   
\item {\bf {\em unit symmetric}} if there exists a unit $u\in \Lambda^{\times}$ such
that $uf$ is symmetric i.e.\ $\iota (uf) =uf$.
\item {\bf {\em mod unit symmetric}} if 
 $\iota (f)=uf$ for some unit $u\in\Lambda^{\times}$.  
\end{itemize}
The definitions which we have just described are given in descending order of strength. For example, $t^{2}+t+1$ is unit symmetric
but not symmetric.  And $t-1$ is mod unit symmetric but not unit symmetric.


\vspace{5mm}

Now we recall Levine's method for generating link polynomials.

\begin{theo}[Levine, \cite{Le2}]\label{LevinesTheorem} Let $\Delta_{L}$ be the Alexander polynomial of an $n$-component link $L$, and let $\lambda = \Sigma \; c_{I}t^{I}\in \Z [H]$ satisfy the following conditions:
\begin{enumerate}
\item[{\rm 1.}] $\lambda$ is symmetric.
\item[{\rm 2.}] $\lambda (1,\cdots ,1)=1$.
\end{enumerate}
Then there exists an oriented link $L'$ with the same number of components as  $L$ such that 
\[ \Delta_{L'}=\Delta_{L}\cdot \lambda.\] 
\end{theo}

Now we prove an analogue of Levine's Theorem for closed 3-manifolds.  By this we mean the following:
let $M$ be a closed 3-manifold with $b_{1}M=n$ and let $\lambda$ be a  symmetric Laurent polynomial in $n$-variables with non-zero trace.  Then we will prove that
there exists a closed 3-manifold $M'$ with  $b_{1}M'=n$ and whose Alexander polynomial is $\lambda\cdot\Delta_{M}$.

\begin{prop}\label{1generator} 
 
Let $M$ be a 3-manifold with $b_{1}M=n$,
$S$ a simple closed curve in M which is homotopically trivial and $N$ a tubular neighborhood of $S$.
Let $p:\widehat{M}\rightarrow M$ be the universal free abelian cover and let $\widehat{X}=\widehat{M}-p^{-1}(N)$.
Then we have the following isomorphism of $\Lambda$-modules
\[H_{1}(\widehat{X})\cong H_{1}(\widehat{M}) \oplus \Lambda.\]



\end{prop}

\begin{proof}  

Consider the pair $(\widehat{M},\widehat{X})$ and denote $\widehat{N} = p^{-1}(N)$.   By excision on the interior of $\widehat{X}$, we have that
\[ H_{*}(\widehat{M},\widehat{X}) \cong H_{*}(\overline{\widehat{N}}, \partial \overline{\widehat{N}}).\] 
Since the deck group $H$ acts by homeomorphisms on the pairs $(\widehat{M},\widehat{X})$, $(\overline{\widehat{N}}, \partial \overline{\widehat{N}})$,
the excision isomorphism is an isomorphism of $\Lambda$-modules.
If we let $\overline{\widehat{N}}_{0}\supset \partial \overline{\widehat{N}}_{0}$ be a pair of fixed components of $\overline{\widehat{N}}\supset \partial \overline{\widehat{N}}$
(i.e.\ a fixed lift of $\overline{N}\supset \partial \overline{N}$) then we may write
\[  \overline{\widehat{N}} =  \bigsqcup_{h\in H} h( \overline{\widehat{N}}_{0})   \supset \partial \overline{\widehat{N}} = 
\bigsqcup_{h\in H} h( \partial\overline{\widehat{N}}_{0}).
 \]  
On the level of homology groups we have then
\[ H_{k} (\overline{\widehat{N}}, \partial \overline{\widehat{N}})\cong\bigoplus_{h\in H}H_{k}(\overline{N}, \partial \overline{N}) , \]
a direct sum of copies of $H_{k}(\overline{N}, \partial \overline{N})$ indexed by $H$.
By Lefschetz duality, since $\overline{N}- \partial \overline{N}=N$, we have that
\begin{equation}\label{lefschetz} H_{k}(\overline{N}, \partial \overline{N}) \cong H^{3-k}(N)\cong H^{3-k}(\SI^{1}) =   \left\{
                                                                \begin{array}{ll}
                                                                \Z & \text{if $k = 2,3.$} \\
                                                                 0 & \text{otherwise.}
                                                                 \end{array}
                                                    \right.
\end{equation}
We note that a generator of $H_{2}(\overline{N}, \partial \overline{N})$ is given by a disk $D$ whose boundary is a meridian of $N$.
Thus as a $\Lambda$-module,  we have
\[   H_{2} (\overline{\widehat{N}}, \partial \overline{\widehat{N}})\cong \Lambda ,  \]
with generator a disk in $ \overline{\widehat{N}}_{0}$ whose boundary is a meridian.
The generator of $H_{2}(\widehat{M},\widehat{X})\cong H_{2} (\overline{\widehat{N}}, \partial \overline{\widehat{N}}) $ corresponding to the disk generator $\widehat{D}_{0}$ of $H_{2} (\overline{\widehat{N}}, \partial \overline{N})$
is denoted $\nu$.  Note that the boundary of $\nu$ is equal to $\partial \widehat{D}_{0}$.

The long exact sequence of the pair $(\widehat{M},\widehat{X})$ (an exact sequence of $\Lambda$-modules) can now be written
\[ \cdots \longrightarrow H_{2}(\widehat{M},\widehat{X})\cong\Lambda \stackrel{\alpha}{\longrightarrow}
H_{1}(\widehat{X}) \stackrel{\beta}{\longrightarrow} H_{1}(\widehat{M})  \stackrel{\gamma}{\longrightarrow}  H_{1}(\widehat{M}, \widehat{X}) =0\longrightarrow\cdots
\]

Then $\gamma$ is the zero map
and by exactness $\beta$ is surjective, so that:
\[ H_{1}(\widehat{M})\cong H_{1}(\widehat{X})/ {\rm Ker}(\beta). \]
Again by exactness ${\rm Ker }(\beta)= {\rm Im }(\alpha)$ and we obtain a short exact sequence
\begin{equation}\label{shorty}
0 \longrightarrow {\rm Im}(\alpha) \longrightarrow H_{1}(\widehat{X}) \longrightarrow H_{1}(\widehat{M}) \longrightarrow 0.
\end{equation}
Since, as indicated above, $H_{2}(\widehat{M},\widehat{X})\cong\Lambda$, the map $\alpha$ is of the form \[ \alpha : \Lambda \longrightarrow  H_{1}(\widehat{X}).\]
Thus if we can show that $\alpha$ is injective and that the sequence (\ref{shorty}) is split by a $\Lambda$-module map, this
will imply that $H_{1}(\widehat{X})\cong H_{1}(\widehat{M})\oplus \Lambda$.




We begin by showing that $\alpha$ is injective i.e.\ we show that ${\rm Im}(\alpha)\cong \Lambda$.
Note first that ${\rm Im}(\alpha)$ is generated as a $\Lambda$-module by $\mu=\beta (\nu)$, and $\mu$ is in fact the meridian boundary of $\widehat{D}_{0}$.  Suppose that $\mu$ is a torsion element of $H_{1}(\widehat{X})$, i.e $f\mu=0$ for some 
$f=\sum c_{h}h\in\Lambda$.  We will show that $f=0$ i.e.\ that $c_{h}=0$ for all $h\in H$.
Now $f\mu$ bounds a compact surface in $\widehat{X}$ which when included in $\widehat{M}$ can be filled by meridian disks to form a closed orientable surface $\Sigma$ in $\widehat{M}$.
Write $\widehat{S}=\sqcup\widehat{S}_{h}$ where $\widehat{S}_{h}=h\widehat{S}_{0}$ and $\widehat{S}_{0}$ is a fixed lift of $S$ to $\widehat{M}$.  Since $S$ is homotopically trivial, each lift $\widehat{S}_{h}$ is also homotopically trivial.  
Then $\widehat{S}_{h}$ intersects $\Sigma$ $c_{h}$-times with the same orientation. But $\widehat{S}_{h}$ is homologically trivial, so its intersection number with an oriented surface is $0$.  Thus $c_{h}=0$ for all $h$.


Now we want to show that the short exact sequence is split. Since $\widehat{S}_{h}$ is homologically trivial for each $h$, it bounds a compact oriented surface $\Sigma_{h}$ in $\widehat{M}$. Let 
\[ \phi : H_{1}(\widehat{X}) \longrightarrow {\rm Im}(\alpha ) =\Lambda\mu \] 
\[l \longmapsto \sum_{h\in H} (l\cdot \Sigma_{h})h\mu\] where  $l\cdot\Sigma_{h}$ is the intersection number with $\Sigma_{h}$.   Then since $h\mu$ has intersection number one with $\Sigma_{h}$ and $h\mu$ has intersection number zero with any $\Sigma_{h'}$ for $h' \ne h$,  
we have $\phi\circ \gamma (h\mu)=h\mu$.
Thus the exact sequence splits,  proving the Proposition.


\end{proof}


\begin{theo}\label{generalizedLevinestate}
Let $M$ be a closed 3-manifold with $b_{1}M=n$, $H\cong\Z^{n}$ the deck group
of its universal free abelian cover and let $\lambda = \Sigma \; c_{I}t^{I}\in \Z [H]$ be a symmetric Laurent polynomial with trace not equal to $0$.  Then there exists a closed 3-manifold $M'$ with $ b_{1}M'= b_{1}M$ and having Alexander polynomial \[ \Delta_{M'}=\Delta_{M}\cdot  \lambda .\]

\end{theo}

\begin{proof}  Let    $\lambda = \Sigma \; c_{I}t^{I}$ be a symmetric Laurent polynomial.
We start by constructing a simple closed curve $S\subset M$ associated to the polynomial $\lambda$ with which we will modify $M$.  







Let $B\subset M$ be a 3-ball and choose a simple closed curve $S_{0}\subset B$ bounding an embedded disk $D$ in $M$. We will modify $S_{0}$ to obtain a new simple closed curve $S$ that will bound an immersed disk in $M$.
For each term pair $\{ t^{I}, t^{-I}\}$, $I\not=0$, of the polynomial $\lambda$ having non-zero coefficient $c_{I}$, pick
\begin{enumerate}
\item A point $q_{I}$ on $S_{0}$ such that $q_{I}\not=q_{I'}$ if $\{ t^{I}, t^{-I}\}\not= \{ t^{I'}, t^{-I'}\}$.
\item One of the two homology classes associated to $\pm I\in \Z^{n}\cong
H_{1}(M)/{\rm Tor}(H_{1}(M))$, which we denote  $\gamma_{\pm I}$.
\end{enumerate}
Then pick an embedded loop $u_{I}$ in $M$ based at $q_{I}$ such that \[ [u_{I}]=\gamma_{\pm I}. \] We may assume, after an isotopy, that the loops $u_{I}$ are disjoint.  Consider
a segment $\tilde{u}_{I}$ obtained from $u_{I}$ by cutting off a small piece of one of its ends, then thicken this segment to a band $\tilde{u}_{I}\times [0,1]\approx [0,1]\times[0,1]$.
See Figure \ref{Band}.  We assume that we have done this in such a way that
\begin{itemize}
\item

$(\tilde{u}_{I}\times [0,1])\cap D= (\tilde{u}_{I}\times [0,1])\cap S_{0} =\{0\} \times [0,1]$ and that the two end segments $\{0,1\} \times [0,1]$
lie in a small neighborhood of $q_{I}$.
\item Each component $[0,1] \times \{0,1\}$ is a copy of $u_{I}$ in the sense that the union of it with a small arc in $S_{0}$ is isotopic to $u_{I}$. We may again assume that all such bands are disjoint.
\end{itemize}
Finally, modify the end of the band as in Figure \ref{Band} so that its boundary links according to the coefficient $c_{I}$.  
We call the new curve $S$.  Note that $S$ bounds an immersed disc in $M$ and is therefore homotopically 
trivial.  

\begin{figure}[htbp]
\centering
\includegraphics[width=5in]{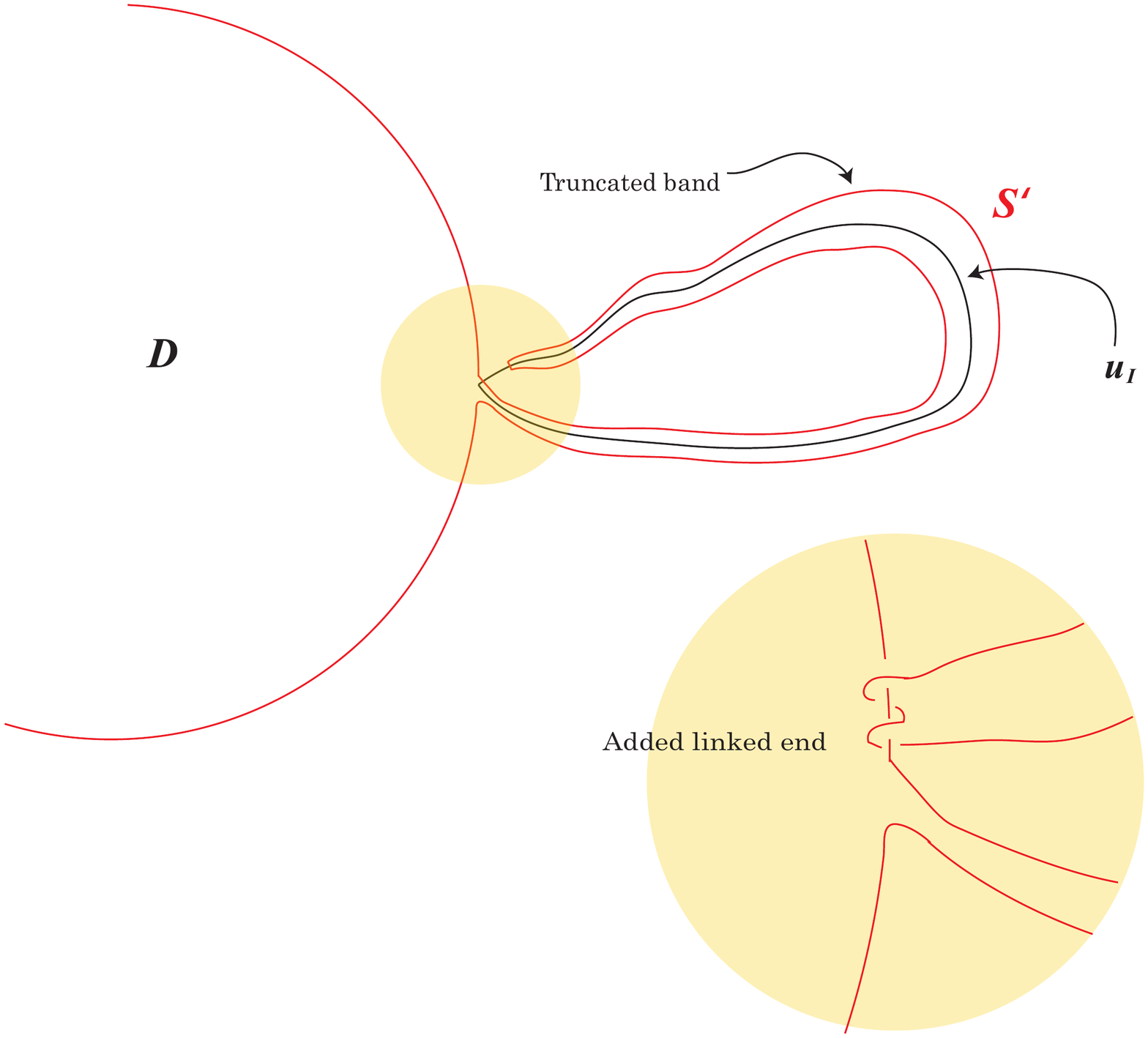}
\caption{Modifying $S_{0}$ to obtain $S$.}
\label{Band}
\end{figure}


Remove a tubular neighborhood $N$  of $S$ and let $X=M-N$. We construct $M'$ by performing an $m$-surgery on $M$ along $S$,
where $m={\rm tr}(\lambda )$, and with respect to
a {\it preferred framing} \[ f=(\ell,\mu)\subset\partial N.\]  Here, $\ell$ is the {\it preferred longitude} -- the one
characterized by defining a homologically trivial element in $X$ -- and $\mu$ is a meridian.
We will show that $\Delta_{M'}=\Delta_{M}\cdot \lambda$.

It will be important for us to have an explicit understanding of the preferred longitude $\ell$.  First, consider a tubular neighborhood of the curve as it appears before the linking step; call that curve $S'$.
See the first image in Figure \ref{Band}.  Choose a preferred longitude  for the original curve $S_{0}$, cut at the beginning
of the band.  Continue this longitude along a pair of homotopic segments lying along tubular neighborhoods
of the band boundaries. The result will be a curve $\ell'$ which is homologically trivial in the complement of a tubular neighborhood of $S'$.

A tubular neighborhood for $S$ is obtained by cutting the tubular neighborhood of $S'$ just constructed and adding a cylinder that links about $S_{0}$.
If we were to continue $\ell'$ without twisting about this cylindrical piece, we obtain a longitude called the {\bf {\em obvious longitude}},
denoted $l$,
which is not homologically trivial.  In fact, $l$ bounds an immersed punctured disk, punctured twice for each linking that has been introduced.  More specifically,
$l$ is homologous in $X$ to
\[    (\sum_{I\not=0} c_{I})\mu   \]
where $\mu$ is a meridian of the tubular neighborhood of $S$.
The preferred longitude of $S$, denoted $\ell$, is therefore obtained from $l$ by introducing a
 {\it pair} of twists in $l$ (opposite in orientation to the direction of the linkings) for each of the  $c_{I}$ linkings coming from $u_{I}$.
 That is,
 \[\ell = l -(\sum_{I\not=0} c_{I})\mu. \]
See Figure \ref{preferredframeS}.

\begin{figure}[htbp]
\centering
\includegraphics[width=5in]{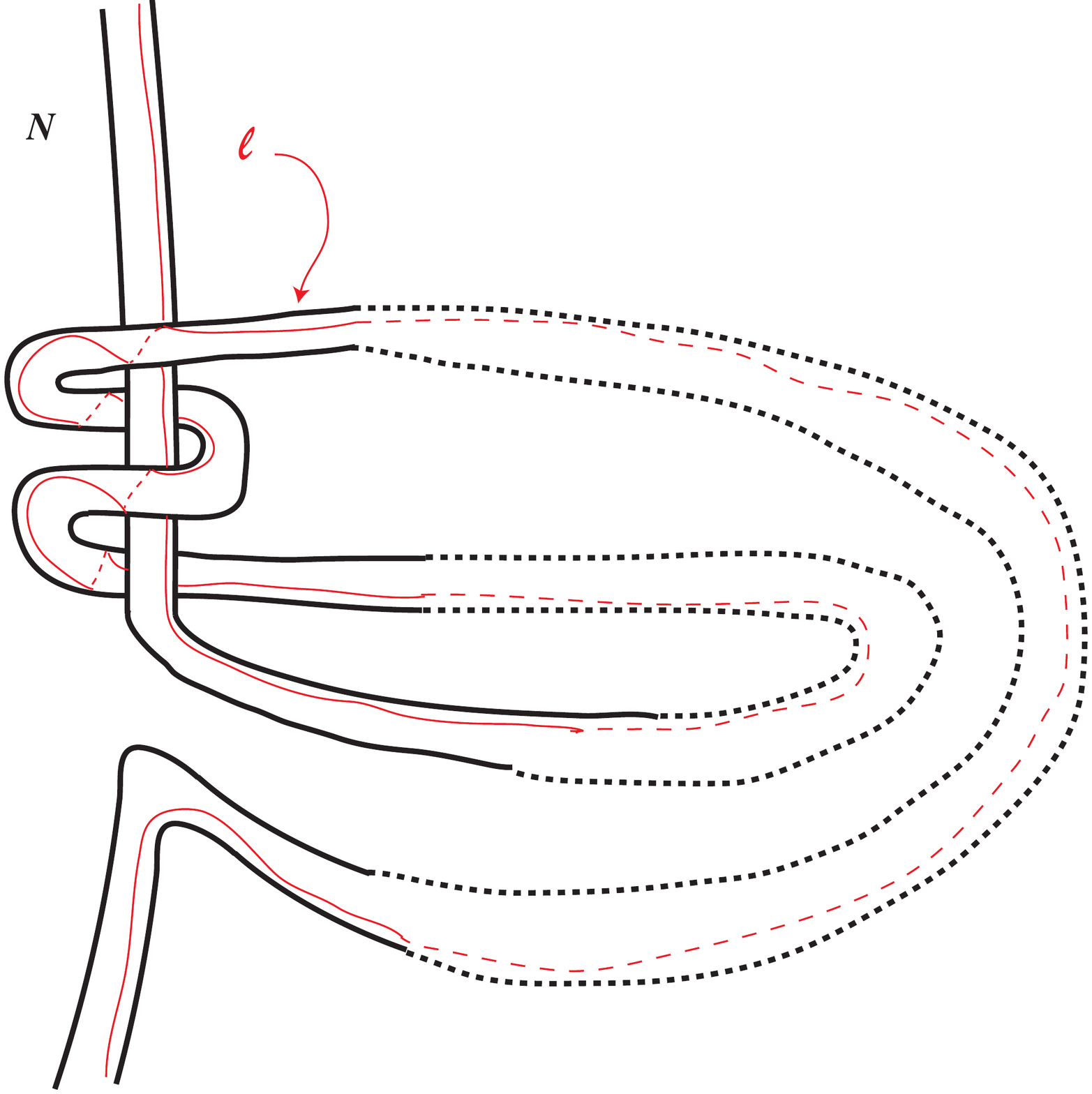}
\caption{The preferred longitude.}
\label{preferredframeS}
\end{figure}

We will now calculate the Alexander module $H_{1}(\widehat{M}')$ as a $\Lambda$-module. Let $\widehat{M}$ be the $\Z^{n}$ cover of $M$, $p$ the covering map so that $\widehat{S}=p^{-1}(S)$ and $\widehat{N}=p^{-1}(N)$.
Let $\widehat{X}=\widehat{M}-\widehat{N}$. By Proposition \ref{1generator} we have that
\[ H_{1}(\widehat{X})\cong H_{1}(\widehat{M}) \oplus \Lambda\]
as $\Lambda$-modules.
Then the only new generator in homology which one obtains by removing $\widehat{N}$ corresponds algebraically to the factor $ \Lambda$.  In particular, this factor as a $ \Lambda$-module is cyclic, generated by an element $\alpha$.

In order to obtain $\widehat{M}'$ from $\widehat{M}$, we will glue in solid tori to $\widehat{X}$, that is, to perform an $m$-surgery on each torus component.  We will
choose a longitude in each torus boundary $\widehat{N}$ of $\widehat{X}$ coming from the lift of the preferred longitude $\ell$
of $N$.  For a fixed boundary component of $\widehat{N}$, we denote this longitude by $\hat{\ell}$ and denote by $\mu$
a meridian.  We let $\hat{l}$ be a lift to  $\widehat{N}$ of the obvious longitude $l$ to the same boundary component containing $\hat{\ell}$.
 Observe that the homology class of $\mu$ also generates the factor $  \Lambda$ occurring
in $H_{1}(\widehat{X})$, so we may assume that $\mu=\alpha$.

The longitude $\hat{l}$ is not homologically trivial (as $l$ was not) owing to the fact that other lifts of $N$ link with $\widehat{N}$.  Instead, it bounds
an immersed punctured disk, in which we have a puncture for each such linking.  See Figure \ref{punctures}.

\begin{figure}[htbp]
\centering
\includegraphics[width=5in]{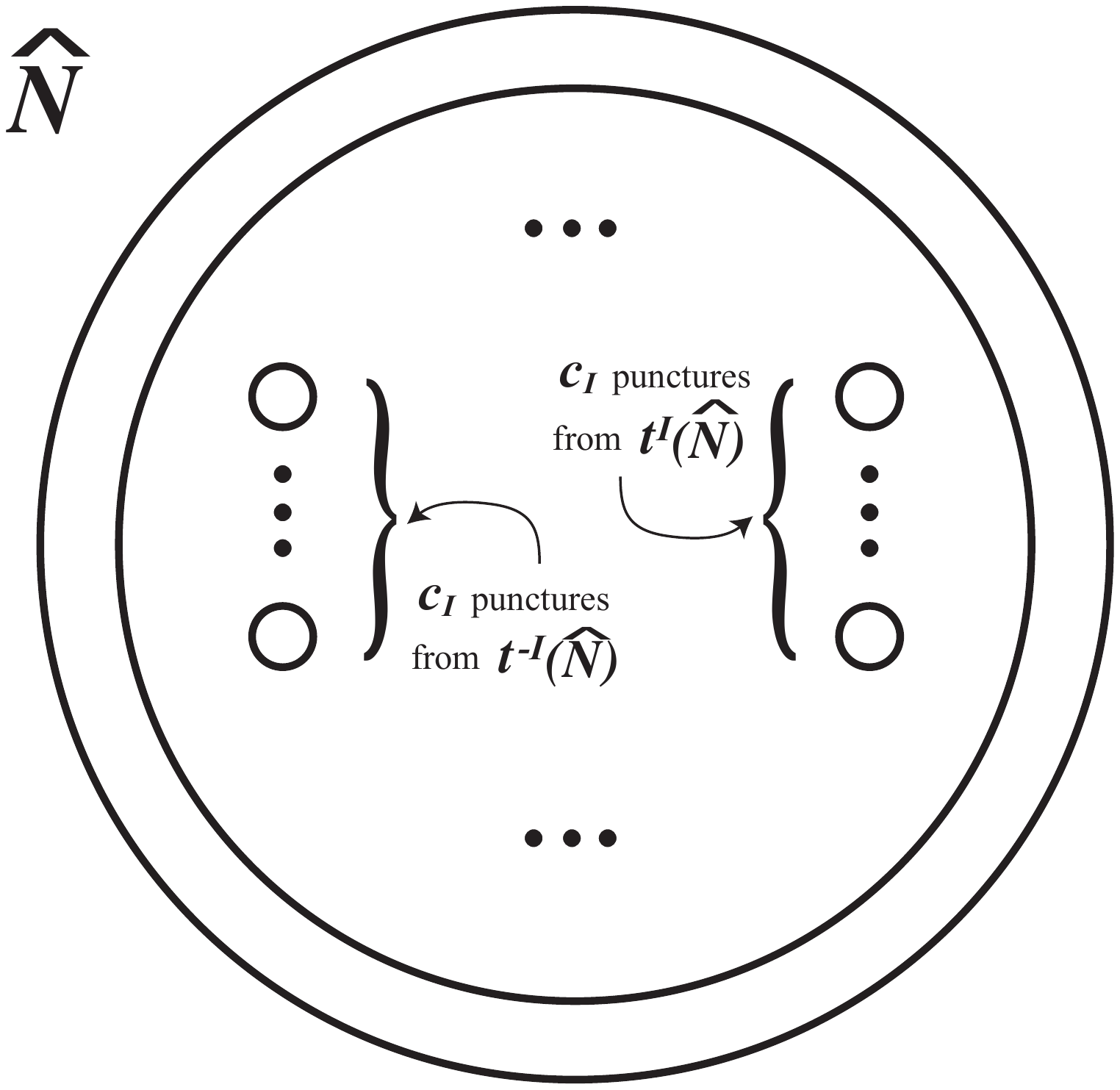}
\caption{The immersed punctured disk.}
\label{punctures}
\end{figure}

Let $\lambda_{0}=\lambda - c_{0}$ so that ${\rm tr}(\lambda_{0})= \sum_{I\not=0} c_{I}$.  Note then
that $\hat{l}=\lambda_{0}\cdot\mu$ i.e.\ is the sum of the punctures.  Then it follows by our construction of $\ell$ from $l$ 
that
\[   \hat{\ell}=(\lambda_{0}-{\rm tr}(\lambda_{0}))\mu = \hat{l}-{\rm tr}(\lambda_{0})\mu .  \]
For the preferred framing $f=(\ell,\mu)\subset \partial X$ an $m$-surgery is given by the formula $\ell+m\mu$.  Thus, in the manifold $M'$, $\ell+m\mu$ is trivial.  This relation produces the relation in $\widehat{M}'$
\[  \hat{\ell}+m\mu = (\lambda_{0}-{\rm tr}(\lambda_{0}) +m)\mu =0.  \]
Thus since $m={\rm tr}(\lambda )$ we have $m-{\rm tr}(\lambda_{0})=c_{0}$ then this relation becomes
\[   \lambda\mu =0.\]

Since this new relation only involves the new generator $\alpha =\mu$, then the presentation of the Alexander
module of $M'$ is of the form
\[ \langle x_{1},\dots ,x_{\alpha}, \mu |r_{1},\dots ,r_{b}, \lambda\mu \rangle   \]
where
$\langle x_{1},\dots ,x_{\alpha}|r_{1},\dots ,r_{b}\rangle$ is the presentation for the Alexander module of $M$.  It then follows
immediately that the presentation matrix for the Alexander module of $M'$ is of the form
\[P' = \left(
                        \begin{array}{cc}
                        P & {\bf 0} \\
                        {\bf 0} & \lambda
                        \end{array}
\right)\]
where ${\bf 0}$ are zero vectors.  It follows that 
$\Delta_{M'} = \Delta_{M}\cdot \lambda$.

\end{proof}

\vspace{5mm}

For low Betti numbers $b_{1}=1, 2$ or $3$, the following closed orientable 
manifolds have 
$\Delta_{M}=1$ (see \cite{Mc}):

\begin{itemize}
\item[-] $\SI^{1}\times\SI^{2}$, $b_{1}=1$.
\item[-] $\text{\sf H}_{3}(\R)/\text{\sf H}_{3}(\Z )$ = Heisenberg manifold, $b_{1}=2$.
\item[-]  $\T^{3}$ = the 3-torus, $b_{1}=3$.
\end{itemize} 

Combining these examples with Theorem \ref{generalizedLevinestate} we have

\begin{coro}\label{sufficiency}  Every symmetric Laurent polynomial in $1$,$2$ or $3$ variables having non-zero trace is the Alexander polynomial of a $3$-manifold with first Betti number $1$,$2$ or $3$.
\end{coro}

\section{Characterization for $b_{1}M=1$}

In this section we prove the  

\begin{coron}  Let $\lambda$ be a Laurent polynomial in $1$ variable.  Then $\lambda =\Delta_{M}$
for some closed $3$-manifold with $b_{1}M=1$ $\Leftrightarrow$ $\lambda$ is unit symmetric and has non-zero trace.
\end{coron}

The sufficiency of the condition ``$\lambda$ is unit symmetric and has non-zero trace'' follows from Corollary \ref{sufficiency}. Therefore we will dedicate this section to proving necessity, which will
be accomplished as follows:
  
\begin{enumerate}
\item[\S \ref{blanchfield}] We recall Blanchfield's symmetry result, which says that the Alexander polynomial of any closed 3-manifold is mod unit symmetric. 
\item [\S\ref{SeifSurf}] We show the existence of a ``Seifert surface'' $\Sigma$  for $M$ and use it to construct the universal infinite cyclic cover $\widehat{M}$.
\item[\S \ref{seifertmat}]  Using $\Sigma$ we define a ``Seifert matrix''  for $M$, and use it to show that  $\Delta_{M}$ is unit symmetric. 
\item[\S \ref{cha}] We prove in Lemma \ref{nonzero} that the trace of $\Delta_{M}$ is not $0$.
\end{enumerate}

\subsection{Blanchfield's Mod Unit Symmetry Theorem}\label{blanchfield}

 Let $M$ be a compact and orientable $3$-manifold with or without boundary. 
In \cite{Blan}, a general symmetry result about Alexander ideals was proved.
 As before, we denote by $\widehat{M}$ the universal free abelian cover and denote by $\Lambda =\Z [H]$ 
the group ring generated by the deck group $H$ of $\widehat{M}\rightarrow M$.  
Recall the automorphism $\iota: \Lambda\rightarrow \Lambda$ defined in \S \ref{generalizedLevinestate}.
We will say that an ideal $\mathfrak{a}\subset\Lambda$ is {\bf {\em symmetric}} if $\iota (\mathfrak{a})=\mathfrak{a}$.




The following theorem is a special case of Corollary 5.6 of \cite{Blan}.

\begin{theo} The principal ideal $(\Delta_{M})$ is symmetric. 
\end{theo}

We have as an immediate corollary:

\begin{coro}\label{modunitsymmetric} Any Alexander polynomial $\Delta_{M}$ of $M$ is mod unit symmetric.
\end{coro}

\begin{proof}  By symmetry $(\iota (\Delta_{M})) = (\Delta_{M})$ so there exists $u\in\Lambda^{\times}$ such that $\iota (\Delta)= u\Delta$ i.e.\ $\Delta$ is mod unit symmetric  \end{proof}



\subsection{An Analogue of the Seifert Surface Construction}\label{SeifSurf}


Let $M$ be closed and orientable with $b_{1}M=1$.  Recall  the map $\psi :\pi_{1}M\rightarrow \Z$ defined as the composition
\begin{diagram}\label{defofpsi}  \pi_{1}M & \rTo^{\text{abelianization}}  &H_{1}(M,\Z ) &\rTo^{\text{projection}} & {\rm Free}(H_{1}(M,\Z ))\cong \Z ,\end{diagram}
where for $A$ an abelian group, ${\rm Free}(A)=A/{\rm Tor}(A)$.  We will usually identify $ {\rm Free}(A)$ as a subgroup of $A$
by choosing a section of the projection $A\rightarrow {\rm Free}(A)$.
For an embedded closed oriented surface $\Sigma\subset M$ and a simple closed oriented curve $\gamma\subset M$ we
denote by $\gamma\cdot \Sigma\in \Z$ the signed intersection number.

\begin{theo}  There exists an oriented, embedded and non separating closed surface $\Sigma\subset M$
such that for all $\gamma\in\pi_{1}M$,
\[  \psi (\gamma ) =   \gamma \cdot \Sigma  .    \]
\end{theo}

\begin{proof}  
Note that $\psi$ is by definition trivial on $[\pi_{1}M,\pi_{1}M]$ so that it induces an element of ${\rm Hom}(H_{1}(M,\Z ),\Z )$: the projection appearing in (\ref{defofpsi}), which is a non trivial homomorphism.
 The intersection pairing version of Poincar\'{e} duality 
provides in particular a non-degenerate pairing 
\[ {\rm Free}( H_{1}(M,\Z ) ) \times  {\rm Free}( H_{2}(M,\Z ) )\longrightarrow \Z \]  
given by the signed intersection number.  

See \cite{GrHa}. N.B. $ H_{2}(M,\Z)\cong H_{1}(M,\Z )$ which is free by the universal coefficient theorem. This means that every homomorphism 
\[ {\rm Free}(H_{1}(M,\Z ))\longrightarrow\Z\] 
 is given by the intersection pairing
with some element of $H_{2}(M,\Z )$.  Thus
we may associate
to $\psi$ an element $[\Sigma ]\in H_{2}(M,\Z )$.   We note that there exists a representative
$\Sigma\in [\Sigma ]$ which is a closed connected embedded surface since $[\Sigma ]$ is of co-dimension 1.
 This surface is orientable because
it represents a non-trivial element of $H_{2}(M,\Z )$;  any closed non-orientable surface has trivial
$H_{2}$ so could not represent a non-trivial element of $H_{2}(M,\Z )$.    Thus we have $\psi (\gamma ) =\gamma\cdot \Sigma$
for all $\gamma\in \pi_{1}M$.
In particular for any curve $\gamma$ with $\psi (\gamma ) =1$ we have  $\gamma\cdot\Sigma =1$.  This, along with the fact that $\Sigma$ is orientable, implies that
$\Sigma$ is non-separating.  For if $M-\Sigma= Y_{1}\sqcup Y_{2}$ is a disjoint union, then since $\Sigma$ is 2-sided, then after
moving $\gamma$ by an isotopy, $\gamma-\Sigma = \gamma-\{ {\rm point}\}\approx (0,1)$ would connect points of $Y_{1}$ to points of $Y_{2}$, which is impossible.

\end{proof}

Since $\Sigma$ is orientable and non-separating, it has a collar which we denote $C(\Sigma )$.  Let
\[  X = \overline{M - C(\Sigma )} = M - {\rm int}(C(\Sigma )).\]

Note that $X$ has two boundary components $\Sigma^{-}$ and $\Sigma^{+}$.


\begin{prop}\label{freepart} If $\Sigma$ is of genus $g$ then
\[  H_{1}X \cong H_{1}\Sigma \oplus {\rm Tor}(H_{1}X)\cong \Z^{2g}\oplus  {\rm Tor}(H_{1}X). \]
In other words, ${\rm Free}(H_{1}X)\cong\Z^{2g}$.
\end{prop} 

\begin{proof}  Observe that the statement of the Proposition is equivalent to showing that
\[  {\rm dim} H_{1}(X;\Q ) =2g. \]  For the remainder of the proof, all homology and cohomology will be with $\Q$ coefficients.
Next, we have
\[   H_{1}(X) \cong H^{2}(M, {\rm int} (C(\Sigma ))   ) \cong H^{2}(M, \Sigma )\]
where the first isomorphism is by Lefschetz duality (see \cite{Sp}, page 297) and the second isomorphism follows because $\Sigma$ is a
deformation retract of ${\rm int} (C(\Sigma ))$.  Since we are working with $\Q$-coefficients,
the universal coefficient theorem implies that
\[  H^{2}(M, \Sigma )\cong H_{2}(M, \Sigma ).\]
(Since $H_{1}M\cong\Q$ is free, ${\rm Ext}(H_{1}M,\Q )=0$ which implies $H^{2}(M, \Sigma )\cong {\rm Hom}(H_{2}(M,\Sigma),\Q )$;
but $H_{2}(M,\Sigma)$ is a $\Q$-vector space, so ${\rm Hom}(H_{2}(M,\Sigma),\Q )\cong H_{2}(M,\Sigma)$.)
So it will be enough to show that $ H_{2}(M, \Sigma )$ has dimension $2g$.

Let us consider the long exact sequence in homology of the pair $(M, \Sigma )$ :
\[  \cdots\longrightarrow H_{2}\Sigma  \stackrel{i_{2}}{\longrightarrow} H_{2}M \stackrel{j_{2}}{\longrightarrow} H_{2}(M, \Sigma )  \stackrel{\partial}{\longrightarrow} H_{1}\Sigma \stackrel{i_{1}}{\longrightarrow} H_{1}M \stackrel{j_{1}}{\longrightarrow} H_{1}(M,\Sigma) \longrightarrow \cdots\]
Notice that 
\[ H_{2}M\cong H^{1}M\cong H_{1}M\cong\Q,\]
where the first isomorphism is by Poincar\'{e} duality, the second by the universal coefficient theorem and the last one because $b_{1}M=1$.
Note that $j_{1}$ is injective, since $\Sigma$ intersects once a generator of ${\rm Free}(H_{1}(M,\Z ))$  and therefore when this generator
is mapped to $H_{1}(M,\Sigma)$, it persists.  

Thus ${\rm Ker}(j_{1})=0 = {\rm Im}(i_{1})$ by exactness, so that $i_{1}$ is the zero map.
Thus ${\rm Ker}(i_{1}) = H_{1}\Sigma = {\rm Im}(\partial )$ again by exactness.  Thus $\partial$ is onto and $H_{2}(M, \Sigma )$ has
dimension $\geq 2g$.  On the other hand, $[\Sigma]$ generates an infinite cyclic subgroup of $H_{2}M$ since it corresponds by duality to $\psi$ which
is a free cohomology class.  So $i_{2}$ is injective and $\Q \cong{\rm Im}(i_{2} ) = {\rm Ker}(j_{2} )$.  Since ${\rm Ker}(j_{2} )\not=0$, it follows since
$H_{2}M\cong\Q$ that we must have that ${\rm Ker}(j_{2} ) = H_{2}M$.  Hence $j_{2}$ is the 0 map, which implies that ${\rm Ker}(\partial ) =0$
i.e. $\partial$ is injective and therefore an isomorphism.
\end{proof}

The argument above can be modified to show that $\partial$ is an isomorphism modulo torsion in $\Z$-coefficients.  More precisely,

\begin{lemm}\label{almostandiso}   The homomorphism 
\[ \partial : H_{2}(M, \Sigma ;\Z)\longrightarrow  H_{1}(\Sigma ;\Z )\]
satisfies 
\begin{enumerate}
\item ${\rm Ker}(\partial)\subset {\rm Tor}(H_{2}(M, \Sigma ;\Z))$. 
\item ${\rm Coker}(\partial )$
is a finite group.
\end{enumerate}
\end{lemm}

\begin{proof}  Consider the long exact sequence of the pair appearing in the proof of Proposition \ref{freepart}, but now with $\Z$-coefficients.  The map
$i_{2}$ is still injective and its image is still free (by how we defined $\Sigma$), but here we can no longer assert that $i_{2}$ is onto.  Nevertheless, this implies
that ${\rm Ker}(j_{2})={\rm Im}(i_{2})$ is free.  We claim that this implies that ${\rm Im}(j_{2})$ is a finite group.  For $H_{2}( M; \Z)\cong \Z\oplus {\rm Tor}(H_{2}(M;\Z))$, since we saw in
the proof of Proposition \ref{freepart} that $H_{2}(M;\Q)\cong H_{2}(M;\Z)\otimes\Q\cong\Q$.  Therefore ${\rm Ker}(j_{2})\subset{\rm Free}(H_{2}(M;\Z))\cong\Z$ and $j_{2}$ induces a map having domain the finite group
$H_{2}(M;\Z)/{\rm Ker}(j_{2})$, so  ${\rm Im}(j_{2})$ is a finite group as claimed.  By exactness, ${\rm Ker}(\partial )$ is finite, which proves (1).  Now since $ H_{2}(M, \Sigma ;\Q)\cong \Q^{2g}$,
$H_{2}(M, \Sigma ;\Z)\cong\Z^{2g}\oplus {\rm Tor}(H_{2}(M, \Sigma ;\Z))$ which by (1) means that  ${\rm Ker}(\partial)\subset  {\rm Tor}(H_{2}(M, \Sigma ;\Z))$.  Therefore $\partial$ restricted to the free part
of $H_{2}(M, \Sigma ;\Z)\cong\Z^{2g}$ maps onto a subgroup of rank $2g$ of $H_{1}(\Sigma ;\Z)\cong\Z^{2g}$.  This implies (2).

\end{proof}

 \marginpar{} 
 \begin{exam}\label{rationalcobordism}
For an example where $\partial$ is not surjective, consider the case
where $M$ is the mapping torus 
\[ \T_{A} = \T^{2}\times [0,1]/\sim_{A},\quad (x,0)\sim_{A} (Ax,1)\] associated
to the hyperbolic matrix
 \[ A = \left(\begin{array}{cc}
                                               3 & 2 \\
                                               1 & 1
                                     \end{array}\right).\]
 Since there is no simple closed curve $c\subset\T^{2}$ such that $A^{n}(c)$
 is isotopic to $c$ for some $n$, we have $b_{1}  \T_{A}=1$.  The ``Seifert
 surface" $\Sigma$ in this case is the image of $\T^{2}\times \{ 0 \}$
 in $\T_{A}$.   On the other hand, we know
(see \cite{Mc}) that 
\[ \Delta_{\T_{A}}(t) = {\rm det}(A-It) =  t^{2}-4t+1.  \]
By Lemma \ref{nonzero} of \S \ref{cha} below, the order of ${\rm Tor}(H_{1}\T_{A} )$
is $|\Delta_{\T_{A}}(1)|=2$.  The map $i_{1}:H_{1}\T^{2}\rightarrow H_{1}\T_{A}$, which is induced by the identification $\T^{2}\approx\Sigma\subset \T_{A}$
 has image ${\rm Tor}(H_{1}\T_{A}) \cong\Z/2\Z$: indeed, due to the defining identifications
 if we denote by $x=(1,0),y=(0,1)$ the basis of $\Z^{2}=H_{1}\T^{2}$, then $i_{1}(x) =x=A(x)= 3x+2y$
 and $i_{1}(y)= y=A(y)=x+y$ which implies that $2(x+y)=0$ and $x=0$ and therefore $2y=0$.
 In particular, ${\rm Ker}(i_{1})={\rm Im}(\partial )$
is a proper subgroup of $ H_{1}\T^{2}$.  Therefore, $\partial $ is not onto.      
 \end{exam}

We now use the surface $\Sigma$ to construct the universal infinite cyclic cover of $M$.
Take a countable collection $\{ X_{i}\}$, $i\in\Z$, of copies of $X$, and glue them such that $\Sigma^{+}_{i-1}$ is identified with $\Sigma^{-}_{i}$
by the ``re-gluing" homeomorphisms.  Denote the result $\widehat{M}$.

\begin{theo}  $\widehat{M}$ is a universal infinite cyclic cover of $M$.
\end{theo}

\begin{proof}  We show that there exists an infinite cyclic covering $p:\widehat{M}\rightarrow M$.  Each point in $X_{i}$ is mapped
to its counterpart in $X\subset M$.  Now we map $X$ onto $\overline{X} = X/\sim$ where $x\sim y$ if $x,y\in\partial X$ correspond.  The gluing
used to define $\widehat{M}$ is compatible with the relation $\sim$ so we get a covering map
$\widehat{M}\rightarrow \overline{X}\approx M$ which is infinite cyclic.
\end{proof}

\subsection{An Analogue of the Seifert Matrix} \label{seifertmat}

A key fact which we will need in order to prove that $\Delta_{M}$ is unit symmetric is the existence of a basis
of the homology of $X$ dual to a given basis of the homology of $\Sigma$.  However,
in view of Lemma \ref{almostandiso}, 
we will not be able to do this for integral homology, as one
does for knot complements \cite{Ro}.   In order to address this complication, we will work instead with homology with $\Q$-coefficients.

We begin with a notion of linking number valid in $M$.  Consider disjoint oriented simple closed curves $l_{1},l_{2}\subset M$ whose integral homology classes belong to
${\rm Tor}(H_{1}(M;\Z))$.  Then there exists integers $n_{1},n_{2}$ such that $[n_{1}l_{1}] = 0 = [n_{2}l_{2}]$, and therefore there exist compact and oriented surfaces
$S_{1}, S_{2}\subset M$ with $\partial S_{1} = n_{1}l_{1}$, $\partial S_{2} = n_{2}l_{2}$.

\begin{defi}\label{linked}  The {\bf linking number} of $l_{1}$ with $l_{2}$ is 
\[   {\rm lk}(l_{1}, l_{2}) =  \frac{ ( n_{2}l_{2})\cdot S_{1}  }{n_{1}n_{2}}=  \frac{ l_{2}\cdot S_{1}  }{n_{1}} \in \Q .  \]
\end{defi}
Since we have divided by $n_{1}n_{2}$, the linking number does not depend on the $n_{i}$ chosen so that $[n_{i}l_{i}]=0$.  

\begin{note}   We have ${\rm lk}(l_{1}, l_{2}) =- {\rm lk}(l_{2}, l_{1})$, just as
in the case of the classical linking number.
\end{note}
The linking number defined here is for pairs of simple closed curves $l_{1}$ and $l_{2}$, and only depends
on the isotopy type of $l_{1}\cup l_{2}$.   We can in fact extend it bi-linearly to rational multiples $q_{1}l_{1}$, $q_{2}l_{2}$ by the formula
\[  {\rm lk}(q_{1}l_{1}, q_{2}l_{2}) := q_{1}q_{2}\cdot {\rm lk}(l_{1}, l_{2}) .\]




We now choose special generating sets for the homology of the boundary components of $X$.

Fix bases $\{ a_{i}\}$, $\{ a_{i}^\pm\}$ of $H_{1}(\Sigma;\Z )$, $H_{1}(\Sigma^\pm;\Z )$; where $\{ a_{i}^\pm\}$ is a push-off in the $\pm$ normal direction of $\{ a_{i}\}$. When viewed in $M$, they give elements of Tor$(H_{1}M)$. 
We can choose $n\in\Z$ so that 
$\{na_{i}\}$, $\{na_{i}^{\pm}\}$ are homologous to $0$ in $H_{1}(M;\Z )$ and not just torsion (for example we could take
$n=\prod n_{i}$).

Now let $\{\bar{a}_{i}\}$, $\{ \bar{a}_{i}^{\pm}\}$ be $\{ na_{i}\}$, $\{ n{a}_{i}^{\pm}\}$, these elements give bases of $H_{1}(\Sigma;\Q )$, $H_{1}(\Sigma^{\pm};\Q)$. 

We define a square matrix $V_{\Q}=(\bar{v}_{ij})$ by
\[    \bar{v}_{ij} = {\rm lk} ( \bar{a}_{i}^{+}, \bar{a}_{j} ) .          \]
Notice that this is well-defined since these curves are torsion as elements of $H_{1}(M;\Z)$, and therefore
their linking numbers are defined.  

  Also note that
\[ \bar{v}_{ij} =  {\rm lk}( \bar{a}_{i}, \bar{a}^{-}_{j} )  \] 
so that \[ V_{\Q}^{T}=(\bar{v}_{ij}^{T}), \quad \text{where} \quad \bar{v}^{T}_{ij} =  {\rm lk} ( \bar{a}_{j}, \bar{a}^{-}_{i} ).\]  

We now specify a basis of $H_{1}(X;\Q)$ dual to the basis $\{ \bar{a}_{i}\}$ of $H_{1}( \Sigma;\Q )$ with respect to the bilinear pairing ${\rm lk} (\cdot ,\cdot )$.  

\begin{lemm}\label{Xduality}   There exists a  basis $\bar{\beta}_{1},\dots ,\bar{\beta}_{2g}$ of  $H_{1}(X;\Q)$ such that viewed in $M$
\[   {\rm lk}(\bar{a}_{i},\bar{\beta}_{j} ) = \delta_{ij}   \] 
\end{lemm}

\begin{proof}  The map $\partial : H_{2}(M,\Sigma ;\Q )\rightarrow H_{1}(\Sigma;\Q )$ is an isomorphism.  If $\bar{a}_{i}$ is one of the generators of $H_{1}(\Sigma;\Q )$ specified above, then $\partial^{-1}(\bar{a}_{i})$ is a multiple ${n}_{i}S_{i}$ of a surface $S_{i}$. We may assume ${n}_{i}$ are integers by choosing $n$ large enough, we write $\bar{S}_{i}={n}_{i}S_{i}$. Thus we obtain a generating set $\bar{S}_{1},\dots ,\bar{S}_{2g}$ of $H_{2}(M,\Sigma ;\Q )$ in
 which $\bar{S}_{i}$ is the image in $H_{2}(M,\Sigma ;\Q )$ of the class $n_{i}S_{i}\in H_{2}(M,\Sigma ;\Z)$.

But 
\[ H^{2}(M,\Sigma ;\Q )
\cong H_{2}(M,\Sigma ; \Q)\cong \Q^{2g}\]
where the first isomorphism is a consequence of the universal coefficient theorem 
(see Corollary 4, page 244 of \cite{Sp}) and the second isomorphism comes from composition of isomorphisms
\begin{diagram} H_{2}(M,\Sigma ;\Q )& \rTo^{\partial}_{\cong} H_{1}(\Sigma;\Q )\cong \Q^{2g}.\end{diagram}
Therefore, $H^{2}(M,\Sigma ;\Q )$
has a basis dual to $\{\bar{S}_{i}\}$: given by cohomology classes $f_{1},\dots , f_{2g}$ with $f_{i}({n}_{j}S_{j} ) ={\delta}_{ij}$.  Here we are identifying cohomology classes with functionals of homology.

By Lefschetz duality we have $H_{1}(X;\Q )\cong H^{2}(M,\Sigma ;\Q)$. The duality isomorphism is given by the intersection pairing.  Therefore, if $\bar{\beta}\in H_{1}(X;\Q) $ and $\bar{S}\in H_{2}(M,\Sigma;\Q )$ then the $\Q $-Lefschetz duality isomorphism is induced by 
\[ \bar{S} \mapsto \bar{\beta}\cdot \bar{S} .\]
In particular, if we let $\bar{\beta}_{i}$ be such that the above function coincides with $f_{i}$, then we have 
\[ \bar{\beta}_{i}\cdot \bar{S}_{j}=f_{i}(\bar{S}_{j} )={\delta}_{ij}.\]  But $\partial (\bar{S}_{j})=\bar{a}_{j}$.
This implies that $ {\rm lk}(\bar{\beta}_{i}, \bar{a}_{j})={\delta}_{ij}$.
\end{proof}

Recall the construction of $\widehat{M}$ given at the end of \S \ref{SeifSurf}.  The Mayer-Vietoris Theorem applied to $\widehat{M}$ shows us that as a $\Lambda$-module
\[ H_{1}(\widehat{M};\Z ) \cong (H_{1}(X;\Z )\otimes\Lambda )/{\rm relations} \] where the relations are given by the gluing identifications 
\begin{align}\label{system}  a_{i}^{-} - ta_{i}^{+}=0,\quad i=1,\dots ,2g  
\end{align}
plus the torsion relations (which do not involve $t$):
\[  m_{1}\mu_{1} =0,\dots , m_{k}\mu_{k} =0   \]
where the $\mu_{i}$ are generators of the torsion of $H_{1}(X;\Z )$.  Hence the Alexander matrix has the form
\[   \left(  \begin{array}{cccc}
A(t) & 0 &\cdots & 0 \\
0 & m_{1} & \cdots & 0 \\
\vdots &  \vdots & \ddots & \vdots \\
0 & 0 & \cdots &m_{k}
\end{array}
    \right)   \]
    where $A(t)$ is a $2g\times 2g$ matrix corresponding to the system (\ref{system}).
    
    Let $\Lambda_{\Q}$ be the group ring $\Q [t^{\pm 1}]$ with coefficients in $\Q$. Then $H_{1}(\widehat{M};\Q )$
    is a $\Lambda_{\Q}$-module, and its presentation is given by
    \[ H_{1}(\widehat{M} ;\Q ) \cong (H_{1}(X;\Z )\otimes\Lambda_{\Q}) /{\rm relations} \] 
   where the relations are 
   \[  na_{i}^{-} - tna_{i}^{+} ,\quad i=1,\dots ,2g,\] 
   or in other words
    \[  \bar{a}_{i}^{-} - t\bar{a}_{i}^{+} ,\quad i=1,\dots ,2g.\] 
    Thus these are the only relations we have and the $\Q$ Alexander matrix is 
  \[A_{\Q}(t)=nA(t).\]  
  
  Recall that if $f(t) = \sum_{i=m}^{n}b_{i}t^{i}$ is a Laurent polynomial, where $b_{m},b_{n}\not=0$, then the degree is defined ${\rm deg}(f) = n-m$.  This notion
of degree is invariant with respect to multiplication by units.
  If we let $\Delta_{\Q}(t)$ be the ``Alexander polynomial'' associated to $A_{\Q}(t)$, then 
  \[ \Delta_{\Q}(t) := det (A_{\Q}(t))=
\frac{n^{2g}}{m_{1}...m_{k}}\Delta (t).\]
   We can see that $\Delta (t)$ and $\Delta_{\Q}(t)$
  have the same degree.  
 Our strategy will be to show that $\Delta_{\Q}(t)$ has even degree.

\begin{theo}\label{seifertmatrix}   $V_{\Q} - tV_{\Q}^{\sc T}=A_{\Q}(t)$.
\end{theo}

\begin{proof}  
 $H_{1}(X;\Q )$ has generators the $\bar{\beta}_{i}$ and we may write therefore
\[ \bar{a}_{i}^{\pm} = \sum c_{ij}^{\pm}\bar{\beta}_{j } .  \]
If we take the linking number both sides of this equation with $\bar{a}_{j}$, we get by Lemma \ref{Xduality} that
\[ c_{ij}^{\pm} = -{\rm lk} ( \bar{a}_{i}^{\pm},  \bar{a}_{j} ).\] 
It follows that the relations (\ref{system}) may be re-written
\[   \sum {\rm lk} ( \bar{a}_{i}^{-}, \bar{a}_{j} ) \bar{\beta}_{j} - t   \sum {\rm lk} ( \bar{a}_{i}^{+},  \bar{a}_{j} ) \bar{\beta}_{j}=0. \]
By our definition of $V_{\Q}$ and our identification of $V_{\Q}^{T}$ the relations (\ref{system}) may be re-written
\[\sum (\bar{v}_{ij}-t\bar{v}_{ji})\bar{\beta}_{j} =0, \quad i=1,\dots ,2g.\]
So
\[ A_{\Q}(t) = V_{\Q}-tV_{\Q}^{T},\] as claimed.
\end{proof}

\subsection{Characterization}\label{cha}

We begin with the following corollary of Theorem \ref{seifertmatrix}:

\begin{coro}\label{modpsymm}  $\Delta_{\Q}(t)$ is unit symmetric and in particular is of even degree.
\end{coro}

\begin{proof}  By Theorem \ref{seifertmatrix} we have
\[\Delta_{\Q}(t) = \det (V_{\Q}-tV^{T}_{\Q}) =\det (V_{\Q}^{T}-tV_{\Q}) \]
and therefore
\[\Delta_{\Q}(t^{-1}) = \det (V_{\Q}-t^{-1}V^{T}_{\Q}) =\det ((-t^{-1})\cdot( V^{T}_{\Q}-tV_{\Q})) =t^{-2g}\Delta_{\Q}(t) .\]
Then $f(t)=t^{-g}\Delta_{\Q}(t)$ is symmetric so $\Delta_{\Q}(t)$ is unit symmetric.  The last statement follows since odd degree polynomials cannot
be unit symmetric (they can be at most mod unit symmetric).
\end{proof}

\begin{theo}\label{unitsymmthm}  The Alexander polynomial $\Delta(t)$ of $M$ is unit symmetric.
\end{theo}

\begin{proof}  By Corollary \ref{modpsymm} $\Delta (t)$ is of even degree.
By Blanchfield we know that any $\Delta_{M}(t)$ is mod unit symmetric.  Since the degree is even, this implies
that it is unit symmetric.
\end{proof}

\begin{lemm}\label{nonzero} For $b_{1}M=1$ the trace of the Alexander polynomial $\Delta_{M}$ is non-zero and its absolute value is equal to the order
of ${\rm Tor}H_{1}(M;\Z)$.
\end{lemm}

\begin{proof}
Let $\widehat{P}(t)$ be a presentation matrix for the $\Lambda$-module $H_{1}(\widehat{M})$, then $P:= \widehat{P}(1)$ is a presentation matrix
for $G$ = $p_{\ast}(H_{1}\widehat{M}) $ = the image of $H_{1}\widehat{M}$ in $H_{1}M$ by the covering map $p:\widehat{M}\rightarrow M$.   
Since $\widehat{M}$ is an infinite cyclic covering of $M$, we have the exact sequence

\[ 1\longrightarrow p_{\ast}(\pi_{1}(\widehat{M}))\subset \pi_{1}(M) \longrightarrow \Z \longrightarrow 0 . \]
By abelianizing $\pi_{1}M$, we obtain the sequence
\[   0\longrightarrow G\subset H_{1}(M) \longrightarrow \Z \longrightarrow 0   \]
which is exact: here we are using the fact that 
\begin{enumerate}
\item The image of the subgroup $p_{\ast}(\pi_{1}\widehat{M}))<\pi_{1}(M)$
by the abelianization map $\pi_{1}(M)\rightarrow H_{1}(M)$ is $G$.
\item The epimorphism $\psi :\pi_{1}(M)\rightarrow\Z$
satisfies ${\rm Ker}(\psi )\supset [\pi_{1}(X), \pi_{1}(X)]$.
\end{enumerate}

  Thus $H_{1}(M)\cong G\oplus\Z$ as abelian groups
 so that ${\rm Tor}(G)={\rm Tor}(H_{1}(M))$.  But $G$ is a finitely generated abelian group and so is isomorphic to $\Z^{r}\oplus \Z/n_{1}\Z\oplus\cdots \oplus \Z/n_{k}\Z$ for integers $r, n_{1},\dots n_{k}$, and $P$ must be equivalent to the diagonal presentation matrix ${\rm diag}(n_{1} ,\dots , n_{k})$.  Then
\[  \Delta_{M} (1) = {\rm det}(P)  = n_{1}\cdots n_{k} = | {\rm Tor}(H_{1}(M))  |. \]

\end{proof}

We can now conclude with the

\begin{proof}[Proof (Characterization Theorem)]  Immediate from Theorem \ref{SeifSurf} and Lemma \ref{nonzero}.
\end{proof}

\section{Manifolds with $b_{1}>1$}\label{>1}

In this chapter we consider Alexander polynomials of closed 3-manifolds with $b_{1}>1$.

\subsection{Manifolds with $b_{1}=2,3$}\label{b1=2,3}

As mentioned in the Introduction, the following closed 3-manifolds $M$ have
$\Delta_{M}=1$:
\begin{itemize}
\item[-] $\text{\sf H}_{3}(\R)/\text{\sf H}_{3}(\Z )$ = Heisenberg manifold \cite{Mc}, $b_{1}=2$.
\item[-]  $\T^{3}$ = the 3-torus, $b_{1}=3$.
\end{itemize}

Applying the generalized Levine's  theorem to the above examples, we have the following corollary:

\begin{coro} Let $M$ be a closed $3$-manifold with first Betti number 2 or 3. Then the set of symmetric Laurent polynomials in $2$ or $3$ variables with ${\rm tr} ( \lambda ) \ne 0$, is contained in the set of Alexander polynomials $\Delta_{M}$ with first Betti number equal to $2$ or $3$. \end{coro}

\subsection{Manifolds with $b_{1}\geq 4$}\label{b1>1}

 In this section we will prove 
 
 \begin{theo}\label{b1grt3impDeltnot1} $\Delta_{M}\not =1$ for any closed $3$-manifold with $b_{1}M\geq 4$.
 \end{theo}  

The proof of this theorem requires several facts which we summarize now.
Let $p$ be a prime and $\F_{p} =\Z/p\Z$. 
Recall that for a manifold $N$, the mod $p$ Betti numbers are defined
 \[  b_{k}(N;\F_{p} ) := {\rm rank}( H_{k}(N;\F_{p})).  \]
We will use the abbreviated notation  
\[  d_{p}(N) :=b_{1}(N;\F_{p} )\]
for the first mod $p$ Betti number.

\vspace{3mm}

  \begin{itemize} 
\item[\bf Fact 1.]  Let $\tilde{M}_{p}$ be the finite abelian cover of $M$ associated to the epimorphism
 \[ \psi_{p}:\pi_{1}M\longrightarrow H_{1}(M;\Z) \longrightarrow H_{1}(M;\F_{p} ) .\]
Let $r= d_{p}(M) $.  Using an inequality of Shalen and Wagreich \cite{SW},
we will deduce in \S \ref{SWsec} that
 \begin{align*} d_{p}(\tilde{M}_{p} ) \geq  \dbinom{r}{2} .\end{align*}

 \item[\bf Fact 2.]  Suppose that $\Delta_{M}=1$ and let $M'\rightarrow M$ be a finite abelian
 cover with deck group $\F_{p_{1}}\oplus\cdots \oplus \F_{p_{k}}$, $p_{1},\dots , p_{k}$ primes.  Then
 \begin{itemize}  
 \item[a.] The torsion subgroup of $H_{1}(M';\Z )$ is trivial. 
 \item[b.] $b_{1}M'=b_{1}M$.  This is a consequence of an equality of E. Hironaka \cite{Hiro}.
 \end{itemize}
These statements will be proved in \S \ref{b1grt4sec}.
\end{itemize}

\vspace{3mm}
  
  Assuming for the moment the above facts, we can now give the 
  
  \begin{proof}[Proof of Theorem \ref{b1grt3impDeltnot1}] We consider a cover $\tilde{M}_{p}$ as in Fact 1 with $p$ prime to the order of the torsion subgroup of $H_{1}(M)$, so that $ d_{p}(M) =b_{1}M$ and
  the deck group of $\psi_{p}$ is $\F_{p}^{b_{1}M}$.  Taking $M' = \tilde{M}_{p}$, it follows from Fact 2. (parts a. and b.)  that \[d_{p}(\tilde{M}_{p} )=d_{p}(M)= r.\]  But 
the inequality in Fact 1.\  is satisfied only for $r\leq 3$, since $r<\dbinom{r}{2}$ for all $r\geq 4$.
  \end{proof}
 
 \subsection{$\F_{p}$-Homology and Finite Cyclic Covers}\label{SWsec}
 
 In this section we will derive the inequality of Fact 1. above. 
 
 Fix a prime $p$ and let $G$ be a group.  If $A<G$ is a subgroup then we define 
   \[G\# A = [G,A]A^{p} := \langle [g,a]b^{p} \mid g \in G,a,b\in A\rangle ,\]
   where $[g,a]=gag^{-1}a^{-1}$ and $\langle X\rangle$  means the group generated by $X$.  
      
  We note that if $A\vartriangleleft G$ then $G\# A\vartriangleleft A$ and $A/(G\# A)$ is 
   an elementary abelian $p$-group i.e\ a direct sum
   of copies of $\F_{p}$.  
   The {\bf mod $p$ lower central series} $\{ G_{i}\}$ of $G$ is defined by 
   \[ G_{i+1}=G\# G_{i}\] where $G_{0}=G$.  By the above comments
   we have $G_{i+1}\vartriangleleft G_{i}$ and $G_{i}/G_{i+1}$ is an elementary abelian $p$-group for all $i=0,1,2,\dots$.
 See \cite{St}.

 Let $\Gamma=\pi_{1}M$ where $M$ is an orientable closed 3-manifold, and let $\{\Gamma_{i}\}$
 be its mod $p$ lower central series.  Let $r={\rm rank} (\Gamma/\Gamma_{1})$.   The following result
 appears as Lemma 1.3 in \cite{SW}:

  \begin{theo}[Shalen and Wagreich] $ {\rm rank}(\Gamma_{1}/ \Gamma_{2}) \geq  \dbinom{r}{2} $.
\end{theo}

\begin{proof}  See \cite{SW} or \cite{La}.
\end{proof}

We now give a geometric interpretation of this theorem in terms of first Betti numbers.  We start by noting that 
\begin{prop}
  $\Gamma/\Gamma_{1} \cong H_{1}(M;\F_{p}).$
  \end{prop}

\begin{proof}  This will follow from showing that $\Gamma_{1}$ is the kernel of the projection $\psi_{p}:\Gamma \rightarrow H_{1}(M;\F_{p})$.  First, note that $\psi_{p}$ is the composition
 \[ \Gamma \longrightarrow H_{1}(M;\Z) \longrightarrow H_{1}(M;\F_{p}),\] and the image of $\Gamma_{1}=\Gamma\#\Gamma$ by the first
map in the composition is $p H_{1}(M;\Z)$ which is the kernel of the second map.  Thus $\Gamma_{1}\subset {\rm Ker}(\psi_{p})$. On the other hand, any element $\gamma\in  {\rm Ker}(\psi_{p})$ must belong to a coset of the form $b^{p}[\Gamma ,\Gamma ]$, and thus $\gamma\in\Gamma_{1}$.
\end{proof}

Let $\tilde{M}_{p}\rightarrow M$ be the finite abelian cover associated to the projection
$\psi_{p}: \Gamma\rightarrow H_{1}(M;\F_{p})$.
We have $\pi_{1}\tilde{M}_{p}\cong \Gamma_{1}$ and
\[ {\rm Deck}(\tilde{M}_{p}/M)\cong \Gamma/\Gamma_{1} \cong H_{1}(M;\F_{p}) \cong (\F_{p})^{r}. \]  

Recall the following notation that was used in the introduction to this chapter
\[ d_{p}(N):= {\rm rank} \; (H_{1}(N;\F_{p} )).\]
When $N=M$ we have $d_{p}(M)=r$.  We have the following Corollary to the Theorem of Shalen and Wagreich: 

\begin{coro} $d_{p}(\tilde{M}_{p})  \geq \dbinom{r}{2}  $.
\end{coro}

\begin{proof} We will show that $d_{p}(\tilde{M}_{p})  \geq {\rm rank} (\Gamma_{1}/\Gamma_{2} )$.  Applying the analysis
of the previous paragraphs 
to $\tilde{M}_{p} $ in place of $M$, we know that 
\[d_{p}(\tilde{M}_{p}) ={\rm rank} (H_{1}(\tilde{M}_{p};\F_{p})) ={\rm rank} (  \Gamma_{1} /(\Gamma_{1}\#\Gamma_{1})).\]
Now 
\[ \Gamma_{1} /(\Gamma_{1}\#\Gamma_{1}) = \Gamma_{1}/[\Gamma_{1},\Gamma_{1}](\Gamma_{1})^{p} .\]
But $\Gamma_{2} = \Gamma \# \Gamma_{1} = [\Gamma_{1},\Gamma](\Gamma_{1})^{p}$ so that
\[\Gamma_{1}/\Gamma_{2}=\Gamma_{1}/[\Gamma,\Gamma_{1}](\Gamma_{1})^{p}.\]
Thus $\Gamma_{1}/\Gamma_{2}$ is a quotient of $ \Gamma_{1} /(\Gamma_{1}\#\Gamma_{1})$ and 
the Corollary follows.
\end{proof}

\subsection{Closed 3-manifolds with $\Delta= 1$}\label{b1grt4sec}

We first recall the following result which relates the order of torsion in the homology of finite abelian covers to values of the Alexander polynomial:

\begin{theo}\label{TorFinAbCov}  Let $M'\rightarrow M$ be a finite abelian cover lying below the universal
free abelian cover $\widehat{M}\rightarrow M$, with deck group
$\F_{p_{1}}\oplus \cdots \oplus \F_{p_{n}}$, where the $p_{i}$ are primes.  
Assume that $\Delta_{M}(t_{1},\dots ,t_{n})$ has no zero of the form $(\rho_{1}^{e_{1}},\dots ,\rho_{n}^{e_{n}}) $ 
where $\rho_{i}$ is a $p_{i}$th root of unity.  Then
\[ | {\rm Tor}(H_{1}M' ) | = \Big| \prod   \Delta_{M}(\rho_{1}^{e_{1}},\dots ,\rho_{n}^{e_{n}}) \Big|  \]
where the product is over all $(e_{1},\dots ,e_{n})$ with $0\leq e_{i}<p_{i}$.
\end{theo}

Theorem \ref{TorFinAbCov} was first proved in the case of knot complements by Fox (see \cite{Gordon}) and stated in the above generality by Turaev (\cite{Turaev}, page 136) but he only provides a proof for cyclic covers.  
A complete proof can be found in \cite{me}.  
We have immediately part a. of Fact 2:

\begin{coro}\label{notor}  Let $M$ be a closed 3-manifold with $\Delta_{M} =1$ and let $M'\longrightarrow M$
be a finite abelian cover.  Then
\[ {\rm Tor}(H_{1}M') =1.\]
In particular, taking $M=M'$,  ${\rm Tor}(H_{1}M) =1$.
\end{coro}

To prove part b. of Fact 2, we will need a formula of E. Hironaka \cite{Hiro}, which we describe in our setting.  Let $M'\rightarrow M$
be a finite cover and assume that $\Delta_{M}=1$.  Let us denote
\begin{itemize}
\item $D=$ the deck group of $M'\rightarrow M$.
\item $\Gamma=  \pi_{1}M$.
\item $\alpha: \Gamma \rightarrow D$ the projection.
\end{itemize}

For any group $G$, the character group is denoted 
\[ \hat{G}={\rm Hom}_{\rm cont}(G,\C^{\ast} )\] 
where ${\rm Hom}_{\rm cont}$ means the group of continuous homomorphisms.  Write $\hat{1}$ for the trivial character. 
We recall that $\hat{G}$ is a topological group. 
In our case, $G=\Gamma$ or $D$, which are discrete groups, so
continuous homomorphisms are just homomorphisms.   
Since $\alpha: \Gamma \rightarrow D$ is an epimorphism, there is an induced inclusion
\[ \hat{\alpha} : \hat{D} \hookrightarrow \hat{\Gamma}. \]

Let $\chi \in \hat{\Gamma}$ and $\Lambda = \Z [t_{1}^{\pm 1},\dots,t_{r}^{\pm 1}]$.  Note that $\chi$ induces a homomorphism of $\Gamma^{\rm ab}\cong\Z^{r}$ (the last isomorphism is by Corollary \ref{notor}).  We may then extend $\chi$ linearly to a ring homomorphism $\chi :\Lambda \rightarrow\C$.

We now describe the formula of E. Hironaka, following \cite{Hiro}.  Before doing so, we remark that the definition of the Alexander polynomial used in \cite{Hiro} is the one formulated using the 
 {\it relative Alexander module} \[ A_{M}^{\rm rel}=H_{1}(\widehat{M},\hat{x}), \quad \hat{x}=p^{-1}(x), x\in M\]
(see \cite{Mc}).   If $P(t_{1},\dots ,t_{r})  $  is a presentation matrix of  $A_{M}^{\rm rel}$, the Alexander polynomial is defined in this setting to be a generator
of the smallest principal ideal containing the ideal generated by the $(r-1)$ minors of $P(t_{1},\dots ,t_{r})$.  A proof of the equivalence of the relative homology definition with the absolute
homology definition can be found in \cite{me} (the equivalence is in fact implicit in Theorem 2.7 of \cite{Ma} as well as Theorem 16.5 of \cite{Turaevbook}).

Now given  $\chi \in \hat{\Gamma}$ let  $P(\chi )$ denote the matrix with complex entries
obtained by evaluating each entry of $P(t_{1},\dots ,t_{r})$ at $\chi$.  
For each $i$, define
\[  V_{i} = \left\{ \chi\in \hat{\Gamma}|\;  {\rm rank}(P(\chi ))<r-i  \right\}.\]
Then Hironaka's formula (see \cite{Hiro},  page 16,  Proposition 2.5.6.) says that 
\[b_{1}M' = b_{1}M + \sum_{i=1}^{r-1}  |  \hat{\alpha} (\hat{D} \setminus\hat{1}) \cap V_{i}  |. \]

\begin{theo}\label{samebetti} Suppose that $\Delta_{M}=1$ and let $M'\rightarrow M$ be any finite abelian cover of $M$ with deck group $D=\F_{p_{1}}\oplus \cdots \oplus \F_{p_{k}}$,
$p_{1},\dots ,p_{k}$ primes.  Then $b_{1}M'=b_{1}M$.
\end{theo}

\begin{proof} We claim that $ |  \hat{\alpha} (\hat{D} \setminus\hat{1}) \cap V_i  |=0$ for all $i=1,\dots ,r-1$.  To do this, it is enough to show 
that 
\[  |  \hat{\alpha} (\hat{D} \setminus\hat{1}) \cap V_1  |=0,\] since $V_{1}\supset \cdots \supset V_{r-1}$.  So we must show that for every character $\chi \in  \hat{\alpha} (\hat{D} \setminus\hat{1})$, 
\[ {\rm rank}\; P(\chi )\geq r-1.\]   Let us suppose not, that there exists a $\chi$ with ${\rm rank}\; P(\chi)< r-1$.  Then every
$r-1$ minor of $P(\chi )$ is $0$.   But
$P(\chi )$ is obtained by evaluating each polynomial appearing in $P$ at 
\[ t_{1} = \rho_{1}^{e_{1}}, \dots,
t_{k} = \rho_{k}^{e_{k}},\] where $\rho_{j}=\exp (2\pi i/p_{j})$ and the exponents $e_{1}, ..., e_{k}$ depend on $\chi$.  This implies that 
\[\Delta( \rho_{1}^{e_{1}},\dots ,  \rho_{k}^{e_{k}}) =0,\] which contradicts
the fact that $\Delta=1$.  Indeed, the greatest common factor of the $(r-1)\times (r-1)$ minors of $P(\chi )$ is 
$1=\Delta( \rho_{1}^{e_{1}},\dots ,  \rho_{k}^{e_{k}})$, so the minors cannot all be $0$.  This contradicts our hypothesis, and
therefore $ {\rm rank}\; P(\chi )\geq r-1$.
\end{proof}

\end{document}